\documentclass[12pt]{amsart}
\usepackage{amsmath,amsthm,amsfonts,amssymb,latexsym}
\usepackage{enumerate}

\begin{document}

\theoremstyle{plain}

\newtheorem{thm}{Theorem}[section]
\newtheorem{lem}[thm]{Lemma}
\newtheorem{pro}[thm]{Proposition}
\newtheorem{hyp}[thm]{Hypotheses}
\newtheorem{cor}[thm]{Corollary}
\newtheorem{que}[thm]{Question}
\newtheorem{rem}[thm]{Remark}
\newtheorem{defi}[thm]{Definition}
\newtheorem{prob}[thm]{Problem}

\newtheorem*{thmMain}{Main Theorem}
\newtheorem*{thmA}{Theorem A}
\newtheorem*{thmB}{Theorem B}
\newtheorem*{thmC}{Theorem C}
\newtheorem*{probB}{Brauer's Problem 12}

\newtheorem*{thmAcl}{Main Theorem$^{*}$}
\newtheorem*{thmBcl}{Theorem B$^{*}$}

\newcommand{\Maxn}{\operatorname{Max_{\textbf{N}}}}
\newcommand{\Syl}{\operatorname{Syl}}
\newcommand{\dl}{\operatorname{dl}}
\newcommand{\Con}{\operatorname{Con}}
\newcommand{\cl}{\operatorname{cl}}
\newcommand{\Stab}{\operatorname{Stab}}
\newcommand{\Aut}{\operatorname{Aut}}
\newcommand{\Ker}{\operatorname{Ker}}
\newcommand{\fl}{\operatorname{fl}}
\newcommand{\Irr}{\operatorname{Irr}}
\newcommand{\SL}{\operatorname{SL}}
\newcommand{\FF}{\mathbb{F}}
\newcommand{\NN}{\mathbb{N}}
\newcommand{\N}{\mathbf{N}}
\newcommand{\C}{\mathbf{C}}
\newcommand{\OO}{\mathbf{O}}
\newcommand{\F}{\mathbf{F}}

\renewcommand{\labelenumi}{\upshape (\roman{enumi})}

\newcommand{\PSL}{\operatorname{PSL}}
\newcommand{\PSU}{\operatorname{PSU}}

\providecommand{\V}{\mathrm{V}}
\providecommand{\E}{\mathrm{E}}
\providecommand{\ir}{\mathrm{Irr_{rv}}}
\providecommand{\Irrr}{\mathrm{Irr_{rv}}}
\providecommand{\re}{\mathrm{Re}}

\def\Z{{\mathbb Z}}
\def\C{{\mathbb C}}
\def\Q{{\mathbb Q}}
\def\irr#1{{\rm Irr}(#1)}
\def\irrv#1{{\rm Irr}_{\rm rv}(#1)}
\def \c#1{{\mathcal #1}}
\def\cent#1#2{{\bf C}_{#1}(#2)}
\def\syl#1#2{{\rm Syl}_#1(#2)}
\def\nor{\triangleleft\,}
\def\oh#1#2{{\bf O}_{#1}(#2)}
\def\Oh#1#2{{\bf O}^{#1}(#2)}
\def\zent#1{{\bf Z}(#1)}
\def\det#1{{\rm det}(#1)}
\def\ker#1{{\rm ker}(#1)}
\def\norm#1#2{{\bf N}_{#1}(#2)}
\def\alt#1{{\rm Alt}(#1)}
\def\iitem#1{\goodbreak\par\noindent{\bf #1}}
   \def \mod#1{\, {\rm mod} \, #1 \, }
\def\sbs{\subseteq}

\def\gc{{\bf GC}}
\def\ngc{{non-{\bf GC}}}
\def\ngcs{{non-{\bf GC}$^*$}}
\newcommand{\notd}{{\!\not{|}}}
\newcommand{\Out}{{\mathrm {Out}}}
\newcommand{\Mult}{{\mathrm {Mult}}}
\newcommand{\Inn}{{\mathrm {Inn}}}
\newcommand{\IBR}{{\mathrm {IBr}}}
\newcommand{\IBRL}{{\mathrm {IBr}}_{\ell}}
\newcommand{\IBRP}{{\mathrm {IBr}}_{p}}
\newcommand{\ord}{{\mathrm {ord}}}
\def\id{\mathop{\mathrm{ id}}\nolimits}
\renewcommand{\Im}{{\mathrm {Im}}}
\newcommand{\Ind}{{\mathrm {Ind}}}
\newcommand{\diag}{{\mathrm {diag}}}
\newcommand{\soc}{{\mathrm {soc}}}
\newcommand{\End}{{\mathrm {End}}}
\newcommand{\sol}{{\mathrm {sol}}}
\newcommand{\Hom}{{\mathrm {Hom}}}
\newcommand{\Mor}{{\mathrm {Mor}}}
\newcommand{\Mat}{{\mathrm {Mat}}}
\def\rank{\mathop{\mathrm{ rank}}\nolimits}
\newcommand{\Tr}{{\mathrm {Tr}}}
\newcommand{\tr}{{\mathrm {tr}}}
\newcommand{\Gal}{{\it Gal}}
\newcommand{\Spec}{{\mathrm {Spec}}}
\newcommand{\ad}{{\mathrm {ad}}}
\newcommand{\Sym}{{\mathrm {Sym}}}
\newcommand{\Char}{{\mathrm {char}}}
\newcommand{\pr}{{\mathrm {pr}}}
\newcommand{\rad}{{\mathrm {rad}}}
\newcommand{\abel}{{\mathrm {abel}}}
\newcommand{\codim}{{\mathrm {codim}}}
\newcommand{\ind}{{\mathrm {ind}}}
\newcommand{\Res}{{\mathrm {Res}}}
\newcommand{\Ann}{{\mathrm {Ann}}}
\newcommand{\Ext}{{\mathrm {Ext}}}
\newcommand{\Alt}{{\mathrm {Alt}}}
\newcommand{\AAA}{{\sf A}}
\newcommand{\SSS}{{\sf S}}
\newcommand{\CC}{{\mathbb C}}
\newcommand{\CB}{{\mathbf C}}
\newcommand{\RR}{{\mathbb R}}
\newcommand{\QQ}{{\mathbb Q}}
\newcommand{\ZZ}{{\mathbb Z}}
\newcommand{\NB}{{\mathbf N}}
\newcommand{\OB}{{\mathbf O}}
\newcommand{\ZB}{{\mathbf Z}}
\newcommand{\EE}{{\mathbb E}}
\newcommand{\PP}{{\mathbb P}}
\newcommand{\GC}{{\mathcal G}}
\newcommand{\HC}{{\mathcal H}}
\newcommand{\GA}{{\mathfrak G}}
\newcommand{\TC}{{\mathcal T}}
\newcommand{\SC}{{\mathcal S}}
\newcommand{\RC}{{\mathcal R}}
\newcommand{\GCD}{\GC^{*}}
\newcommand{\TCD}{\TC^{*}}
\newcommand{\FD}{F^{*}}
\newcommand{\GD}{G^{*}}
\newcommand{\HD}{H^{*}}
\newcommand{\GCF}{\GC^{F}}
\newcommand{\TCF}{\TC^{F}}
\newcommand{\PCF}{\PC^{F}}
\newcommand{\GCDF}{(\GC^{*})^{F^{*}}}
\newcommand{\RGTT}{R^{\GC}_{\TC}(\theta)}
\newcommand{\RGTA}{R^{\GC}_{\TC}(1)}
\newcommand{\Om}{\Omega}
\newcommand{\eps}{\epsilon}
\newcommand{\al}{\alpha}
\newcommand{\chis}{\chi_{s}}
\newcommand{\sigmad}{\sigma^{*}}
\newcommand{\PA}{\boldsymbol{\alpha}}
\newcommand{\gam}{\gamma}
\newcommand{\lam}{\lambda}
\newcommand{\la}{\langle}
\newcommand{\ra}{\rangle}
\newcommand{\hs}{\hat{s}}
\newcommand{\htt}{\hat{t}}
\newcommand{\tn}{\hspace{0.5mm}^{t}\hspace*{-0.2mm}}
\newcommand{\ta}{\hspace{0.5mm}^{2}\hspace*{-0.2mm}}
\newcommand{\tb}{\hspace{0.5mm}^{3}\hspace*{-0.2mm}}
\def\skipa{\vspace{-1.5mm} & \vspace{-1.5mm} & \vspace{-1.5mm}\\}
\newcommand{\tw}[1]{{}^#1\!}
\renewcommand{\mod}{\bmod \,}

\marginparsep-0.5cm

\renewcommand{\thefootnote}{\fnsymbol{footnote}}
\footnotesep6.5pt

\title{Certain irreducible characters over a normal subgroup}

\author{Gabriel Navarro}
\address{Departament d'\`Algebra, Universitat de Val\`encia, 46100 Burjassot, Val\`encia, Spain}
\email{gabriel.navarro@uv.es}
\author{Noelia Rizo}
\address{Departament d'\`Algebra, Universitat de Val\`encia, 46100 Burjassot, Val\`encia, Spain}
\email{noelia.rizo@uv.es}
\thanks{The research of the first  author is supported by the Prometeo/Generalitat Valenciana,
and Proyecto
MTM2013-40464-P. The second author is supported by a Fellowship
FPU of Ministerio de Educaci\'on, Cultura y Deporte}

\keywords{Howlett-Isaacs theorem}

\subjclass[2010]{Primary 20D; Secondary 20C15}

\begin{abstract}
We extend the Howlett-Isaacs theorem
on the solvability of groups of central type
taking into account actions by automorphisms. Then we study
certain induced characters whose constituents have all the
same degree.
 \end{abstract}

\maketitle
\bigskip

\def\C{{\Bbb C}}
\def\qed{{~~~\vrule height .75em width .4em depth .2em}}
\def\irr#1{{\rm Irr}(#1)}
\def\cent#1#2{{\bf C}_{#1}(#2)}
\def\syl#1#2{{\rm Syl}_#1(#2)}
\def\nor{\triangleleft\,}
\def\zent#1{{\bf Z}(#1)}
\def\norm#1#2{{\bf N}_{#1}(#2)}
\def\oh#1#2{{\bf O}_{#1}(#2)}
\def\iitem#1{\goodbreak\par\noindent{\bf #1}}

\def\irrp#1{{\rm Irr}_{p'}(#1)}
  
  \section{Introduction}
  The celebrated Howlett-Isaacs [HI] theorem on groups of central type
  solved a conjecture proposed by Iwahori and Matsumoto in 1964:
  if $Z$ is a normal subgroup of a finite group $G$, $\lambda \in \irr Z$
  is a $G$-invariant complex irreducible character 
  of $Z$, and the induced character $\lambda^G$ is a multiple
  of a single $\chi \in \irr G$, then $G/Z$ is solvable.
  (In this case, it is said that $\lambda$ is {\bf fully ramified} in $G/Z$
  and that $G$ is of {\bf central type} if furthermore $Z=\zent G$.)
  This theorem, proved in 1982, is one of the first applications
  of the Classification of Finite Simple Groups to Representation
  Theory.   Fully ramified characters are essential
  in both Ordinary and Modular Representation Theory.
  \medskip
  
  Our first main result in this note is the following generalization.

  \begin{thmA}
  Suppose that $Z \nor G$, and let $\lambda \in \irr Z$.
  Assume that if $\chi, \psi \in \irr G$
  are irreducible constituents of the induced character $\lambda^G$,
  then there exists $a \in {\rm Aut}(G)$ stabilizing $Z$,
  such that $\chi^a=\psi$. 
  If $T$ is the stabilizer of $\lambda$ in $G$, then
  $T/Z$ is solvable.
  \end{thmA}

  In a different language of projective representations, 
   Theorem A
  was obtained  by R. J. Higgs under some solvability conditions [H].
  His proof  
  is mostly sketched, among other reasons
   because he uses some of the
  arguments
  in [HI] or [LY] (where, as a matter of fact,
   the Iwahori-Matsumoto conjecture was wrongly proven)
  or in some other papers by the author.
  Here, we choose to present a complete proof of Theorem A, in the language of
  character theory, and by doing so we
  shall adapt several arguments in all these papers. We
  would like to acknowledge this  now.
  
  \medskip
   Theorem A is only one case of a more general problem, which  
   seems intractable by now: if all the irreducible characters
  of $G$ over some $G$-invariant $\lambda \in \irr Z$ have the same degree, then $G/Z$ is solvable.
  (See Conjecture 11.1 of [N].)
  \medskip
  
  In the second main result of this note, we study this latter situation
  under some special hypothesis. 
  
  \begin{thmB}
  Suppose $G$ is $\pi$-separable and let
$N = \oh\pi G$. Let $\theta \in \irr N$ be $G$-invariant. Then
 all members of $\irr{G|\theta}$ have equal degrees if and only if
 $G/N$ is an abelian $\pi'$-group.
  \end{thmB}

 As the reader will see, the proof of Theorem B uses a lot of deep machinery. The proof
 that we present here is an improvement by M. Isaacs
 of an earlier version which
 we reproduce it here with his kind permission.

  \section{Transitive Actions}

In general, we follow the notation in [Is].
If $G$ is a finite group, then $\irr G$ is the set of the irreducible
complex characters of $G$.
If $Z \nor G$ and $\lambda \in \irr Z$, then
$\irr{G|\lambda}$ is the set of the irreducible
constituents of the induced character $\lambda^G$.
By Frobenius reciprocity, this is the set of characters $\chi \in \irr G$
such that the restriction $\chi_Z$ contains $\lambda$.

\medskip

\begin{lem}\label{aux}
Suppose that $Z \nor G$, and let
$\lambda \in \irr Z$ be $G$-invariant. Assume that  all characters 
in $\irr{G|\lambda}$ have the same degree $d\lambda(1)$.  Let $P/Z \in \syl p{G/Z}$. Then $d_p \lambda(1)$ is the minimum of $\{ \delta(1) \, |\, \delta \in \irr{P|\lambda}\}$
and $|\irr{P|\lambda}| \le |\irr{G|\lambda}|_p$.
\end{lem}

\medskip

\begin{proof}
By Character Triple Isomorphisms
 (see Chapter 11 of [Is]),
 we may assume that $\lambda(1)=1$.
Write $\irr{G\, |\, \lambda} = \{\chi_j \mid 1 \le j \le s\}$, and
observe
that the multiplicity of $\chi_j$ in $\lambda^G$ is
$\chi_j(1)$.  Since by
hypothesis, all of the degrees $\chi_j(1)$ are equal, we can write
$\lambda^G = d\sum_j\chi_j$, where $d = \chi_j(1)$ for all $j$.
Also, we have $sd^2 = |G:Z|$.
 Write
$\irr{P\, |\, \lambda} = \{\delta_i \mid 1 \le i \le t\}$, and observe
that because $\lambda(1) = 1$, we have
$\lambda^P = \sum_i d_i\delta_i$, where $d_i = \delta_i(1)$ and
$\sum {d_i}^2 = |P:Z|$. We can write
$$
\delta_i^G = \sum_{j=1}^s d_{ij}\chi_j \,,
$$
and it follows that $\delta_i^G(1)$ is a multiple of the common
degree $d$ of the $\chi_j$. Then $d$ divides $|G:P|d_i$, and
hence the $p$-part $d_p$ of $d$ divides $d_i$ for all $i$. We
conclude that $d_p$ divides the greatest common divisor $e$ of
the $d_i$.

We also have that
$$
(\chi_j)_P = \sum_{i=1}^t d_{ij}\delta_i
$$
by Frobenius reciprocity, and thus $e$ divides $\chi_j(1) = d$.
Since $e$ is a $p$-power, we see that $e$ divides $d_p$, and
thus $e = d_p$. Then we have that
$$
|P:Z| = \sum_{i=1}^{t} {d_i}^2 \ge e^2t  = {(d_p)}^2t \,.
$$ 
Taking $p$-parts in  $sd^2 = |G:Z|$, we obtain that $s_p \ge t$. 
\end{proof}
\medskip
The following is a character-theoretical version
of Theorem 1.2 of [H].

\medskip

\begin{thm}\label{actrans}
Suppose that $Z\nor  G$, $\lambda \in \irr Z$ is $G$-invariant, $p$ is a prime
and $P/Z \in \syl p{G/Z}$. Let ${\mathcal A}=\irr{G|\lambda}$ and ${\mathcal B}=\irr{P|\lambda}$.
Suppose that $A$ is a finite group acting on $\mathcal A$ and $\mathcal B$ in such a way that 
$$[(\chi^a)_P, \delta^a]= [\chi_P, \delta]$$
for all $\chi \in \mathcal A$, $\delta \in \mathcal B$ and $a \in A$.
Assume further that $\chi^a(1)=\chi(1)$ for $\chi \in \mathcal A$ and $a \in A$.
Let $B \in \syl pA$.
If $A$ acts transitively on $\mathcal A$, then $B$ acts transitively on $\mathcal B$ and $|{\mathcal A}|_p=|{\mathcal B}|$.
\end{thm}

\medskip

\begin{proof}
Write ${\mathcal A}= \{ \chi_1, \ldots, \chi_ s\}$ and ${\mathcal B}=\{ \delta_1, \ldots, \delta_t\}$.
By hypothesis, we have that
all the  characters in $\mathcal A$ have the same degree $d \lambda(1)$. Hence
$$|G:Z|= d^2s  \, .$$
By Lemma \ref{aux}, we have that  $d_p\lambda(1)$ is the minimum of the degrees in $\mathcal B$ and that $t \le s_p$.
Write
$$(\chi_i)_P= \sum_{j=1}^t d_{ij} \delta_ j \, $$
so that
 $$(\delta_j)^G= \sum_{i=1}^s d_{ij} \chi_i$$ by Frobenius reciprocity.
Let $B$ be a Sylow $p$-subgroup of $A$. 
Let $\delta_j$ be such that $\delta_j(1)=d_p\lambda(1)$.

Now, let $S=B_{\delta_j}$ be the stabilizer of $\delta_j$ in $B$.
We have that $S$ acts on 
the set $\irr{G|\delta_j}$ of irreducible
constituents of $\delta_j^G$. Let ${\mathcal O}_1, \ldots, {\mathcal O}_r$ be the set of $S$-orbits.
Let $\psi_i \in {\mathcal O}_i$. We may write
$$(\delta_j)^G= \sum_{k=1}^r  b_k (\sum_{ \xi \in {\mathcal O}_k} \xi) \, . $$
Hence
$$|G:P|d_p\lambda(1)= \sum_{k=1}^r b_k|{\mathcal O}_k|\psi_k(1)= d \lambda(1)\sum_{k=1}^r b_k|{\mathcal O}_k|$$
and therefore $p$ does not divide
$$\sum_{k=1}^r b_k|{\mathcal O}_k| \,.$$
Therefore there is $k$ such that $b_k  |{\mathcal O}_k|$ is not divisible by $p$.
In particular, we see that there is an irreducible constituent $\psi_k$ of $(\delta_j)^G$ that
is $S$-fixed. Hence $S=B_{\delta_j} \sbs B_{\psi_k} \sbs B$. Also $B_{\psi_k} \sbs R \in \syl p{A_{\psi_k}}$ for some Sylow
$p$-subgroup $R$ of $A_{\psi_k}$.
Since $A$ acts transitively on $\irr{G|\lambda}$, we have that
$s=|A: A_{\psi_k}|$. Thus
$$s_p=|A|_p/|A_{\psi_k}|_p=|A|_p/|R| = |B|/|R| \le |B|/|B_{\psi_k}|
\le |B|/|B_{\delta_j}|=|B:B_{\delta_j}| \le t \le s_p \, .$$
Thus $t=s_p$, $|B:B_{\delta_j}|=t$ and everything follows. 
\end{proof}

\medskip
 
\section{Auxiliary results}
 
\medskip  
 Of course, if $A$ acts by automorphisms on $G$, then
 $A$ also acts on $\irr G$. If $\chi \in \irr G$ and $a \in A$,
 then $\chi^a \in \irr G$ is the unique character satisfying that
 $\chi^a(g^a)=\chi(g)$ for $g \in G$.  
 \medskip

\begin{hyp}\label{hyp}
Suppose that $Z \sbs N \nor G$, where $Z\nor G$.
Let $\lambda \in \irr Z$.
 Suppose that 
if $\tau_i \in \irr{N|\lambda}$ for $i=1,2$, then there exists
$g \in G$
 such that $\tau_1^{g}=\tau_2$.
 \end{hyp}
  
  \medskip
  
  We say in this case that $(G,N, \lambda)$ satisfies
  Hypothesis \ref{hyp}. 
  \medskip

 If $N \nor G$ and $\tau \in \irr N$, then 
 we denote by $I_G(\tau)$, or by $G_\tau$, the stabilizer
 of $\tau$ in $G$. 

\medskip  
 Recall that induction defines a bijection
 $$\irr{I_G(\theta)|\theta} \rightarrow \irr{G|\theta}$$
 by the Clifford correspondence (Theorem (6.11) of [Is]).
\medskip

\begin{lem}\label{propert}
Suppose that $(G,N, \lambda)$ satisfies Hypotheses \ref{hyp}.
Let $Z \sbs K \sbs N$, where $K \nor G$.
Then the following hold.
\begin{enumerate}[(a)]

\item
Let $\tau_i \in \irr{K|\lambda}$ for $i=1,2$. Then there exists
$g \in G $
 such that $\tau_1^g=\tau_2$.
 
\item
Suppose that $L \nor G$ is contained in $K$. Let $\epsilon \in \irr L$.
Suppose that $\gamma_i \in \irr{I_K(\epsilon) |\epsilon}$ are such that
$(\gamma_i)^K$ lie over $\lambda$ for $i=1,2$.
Then there is $g \in I_G(\epsilon) $
such that $\gamma_1^g= \gamma_2$.

\item
Let $\tau \in \irr{K|\lambda}$. 
Let $\gamma_i \in \irr{I_N(\tau)|\tau}$
 for $i=1,2$.
 Then there exists $g \in I_G(\tau)$ such that $\gamma_1^g=\gamma_2$. 
 \end{enumerate}
 \end{lem}
 
 \medskip
 
\begin{proof}

(a)~~Let $\gamma_i \in \irr N$ over $\tau_i$.
 By hypothesis, we have that
  $\gamma_1^{x}=\gamma_2$ for some $x \in G$.
 We have that $\tau_1^{x}$ and $\tau_2$ are under
 $\gamma_2$, so by Clifford's theorem there is $n \in N$ such that 
 $\tau_1^{xn}=\tau_2$. Set $g=xn$.

(b)~~By part (a), there is $g \in G $
such that $((\gamma_1)^K)^{g}=(\gamma_2)^K$.
Now, $\epsilon^{g}$ and $\epsilon$ are under
$(\gamma_2)^K$, so by replacing $g$ by $gk$ for some $k \in K$,
we may assume also that $\epsilon^{g}=\epsilon$.
Then $g \in I_G(\epsilon)$. By the uniqueness in the Clifford
correspondence,
 we deduce that $(\gamma_1)^{g}=\gamma_2$.

 (c)~~By the Clifford correspondence, we have that $(\gamma_i)^N \in \irr{N}$ lie over $\lambda$.
 By hypotheses, there
 is $g \in G$ 
  such that $((\gamma_1)^N)^{g}=
 (\gamma_2)^N$. Now, $\tau^{g}$ and $\tau$
 are $N$-conjugate by Clifford's theorem, so by
 replacing $g$ by $gn$, for some $n \in N$,
 we may assume that $\tau^{g}=\tau$.
 Notice now that  $g \in I_G(\tau)$.
 Also, $\gamma_1^{g}=\gamma_2$,
 by the uniqueness in the Clifford correspondence.
 \end{proof}
 
 \begin{thm}\label{div}
 Assume Hypotheses \ref{hyp}, with $Z \sbs \zent N$.
 Let $U \sbs N$, with $U \nor G$. 
 Suppose that $q$ is a prime dividing $|U|$. Then $q$ divides
 $|Z \cap U|$.
 \end{thm}
 
 \begin{proof}
 Let $K=UZ \nor G$. If $q$ does not divide $|K:Z|=|U:U\cap Z|$ then we are done.
 Let $1 \ne Q/Z \in \syl q{K/Z}$.
  By Lemma \ref{propert}(a), we know that $G_\lambda=I_G(\lambda)$ acts transitively
 on $\irr{K|\lambda}$. By the Frattini argument, we have that $G_\lambda=
 K\norm{G_\lambda}Q$. Notice then that $A=\norm{G_\lambda}Q$
 acts transitively on $\irr{K|\lambda}$. Also $A$ acts on $\irr{Q|\lambda}$
 and $[(\chi^a)_Q, \delta^a]=[\chi_Q, \delta]$ for $a \in A$,
 $\chi \in \irr{K|\lambda}$ and $\delta \in \irr{Q|\lambda}$.
 By Theorem \ref{actrans}, we have that $A$ acts transitively on
 $\irr{Q|\lambda}$.
 
 Suppose now that $q$ does not divide $|Z \cap U|$. Let $\nu=\lambda_{Z \cap U}$.
 Then $\nu$ has a canonical extension $\hat\nu \in \irr{Q\cap U}$
 of $q'$-order.
 By Corollary (4.2) of [Is2], we know that
 restriction defines a natural bijection
 $$\irr{Q|\lambda} \rightarrow \irr{Q\cap U|\nu} \, .$$
 Let $\rho \in \irr{Q|\lambda}$ be such that $\rho_{Q\cap U}=\hat\nu$.
 In particular, $\rho$ is linear. Also $\rho_Z=\lambda$.
 Let $a \in A$. Then $a$ fixes $\lambda$, and therefore $\nu$. Now, $a$ normalizes $Q$
 and $U$, so $a$ normalizes $U \cap Q$. By uniqueness, we have that $(\hat\nu)^a=
 \hat\nu$. Thus $\rho^a=\rho$ by uniqueness.
 Since $A$ acts transitively on $\irr{Q|\lambda}$, it follows that
 $\irr{Q|\lambda}=\{ \rho\}$. Since $\rho_Z=\lambda$, by Gallagher
 Corollary (6.17) of [Is],
 we know that $|\irr{Q|\lambda}|=|\irr{Q/Z}|$. We conclude that $Q=Z$, and this is the final
 contradiction.
  \end{proof}
 
 \medskip

 As in [HI], we shall only use the Classification of Finite Simple Groups
 in the following result. If $X$ is a finite group, recall that $M(X)$
 is the Schur multiplier of $X$.
 
 \medskip
 
 \begin{thm}\label{HI}
 Let $X$ be a non-abelian simple group. Then there exists a prime
 $p$ such that $p$ divides $|X|$, $p$ does not divide $|M(X)|$,
 and there is no solvable subgroup of $X$ having $p$-power index.
 \end{thm}
 
 \begin{proof}
 This is Theorem (2.1) of [HI].
 \end{proof}

 \section{The Glauberman  Correspondence}
 
 The idea to use the Glauberman  correspondence
 in the Iwahori-Matsumoto conjecture appears in [HI].
 As we shall see in the proof of our main theorem,
 we need to do the same here, in a more sophisticated way.
 For the definition and properties of the Glauberman correspondence,
 we refer the reader to Chapter 13 of [Is].
 
 \medskip
 
 We remark now the following.
 If $Q$ is a $p$-group that acts by automorphisms 
 on  a $p'$-group $L$, then the Glauberman correspondence is
 a bijection $^*: {\rm Irr}_Q(L) \rightarrow \irr C$,
 where ${\rm Irr}_Q(L)$ is the set of $Q$-invariant irreducible
 characters of $L$ and $C=\cent LQ$ is the fixed point subgroup.
 Furthermore, for $\chi \in {\rm Irr}_Q(L)$, we have that
 $$\chi_C=e \chi^* + p \Delta \, ,$$
 where $p$ does not divide the integer $e$ and $\Delta$ is
 a character of $C$ (or zero).
 In particular, we easily check that the Glauberman correspondence $^*$
 commutes with the action of ${\rm Gal}(\bar \Q/\Q)$
 and with the action of the group of automorphisms of the semidirect product
 $LQ$ that fix $Q$.
 In particular, we have that $\Q(\chi)=\Q(\chi^*)$.

 \medskip

 The next deep result is key in character theory. Its  proof, in the case
 where $Z=1$, 
 is due to E. C. Dade. (Other proofs are due to L. Puig.)
 The following useful strengthening  
 is due to A. Turull, to whom we thank   for useful conversations on this subject.
 
 \begin{thm}\label{deepglau}
 Suppose that $G$ is a finite group, $LQ \nor G$, where $L \nor G$,
 $(|L|,|Q|)=1$, and $Q$ is a $p$-group for some prime $p$.
 Suppose that $LQ \sbs N \nor G$, and $Z \nor G$,
 is contained in $Q$ and in $\zent N$. Let $\lambda \in \irr Z$.
 Let $H=\norm G Q$ and $C=\cent LQ$.
 Then for every  $\tau \in {\rm Irr}_Q(L)$ there is a bijection
 $$\pi(N,\tau):    \irr{N|\tau} \rightarrow
  \irr{N \cap H|{\tau^*}} \, ,$$
  where $\tau^* \in \irr C$ is the $Q$-Glauberman correspondent
  of $\tau$, 
  such that:
  \begin{enumerate}[(a)]
  
 \item
 For $\gamma \in \irr{N|\tau}$, $h \in H$ we have that
  $$\pi(N, \tau^{h})(\gamma^{h})=
 (\pi(N, \tau)(\gamma))^{h} \, .$$
 
 \item
 $\rho \in \irr{N|\tau}$ lies over $\lambda$ if and only if $\pi(N,\tau)(\rho)$
 lies
 over $\lambda$.
 
 \end{enumerate}
 \end{thm}

 \medskip
 
\begin{proof}
It follows from the proofs of Theorem 7.12 of [T1]
and Theorem 6.5 of
 [T2]. Specifically, we make $\psi=\theta$
 in Theorem 7.12 of [T1],
 and  $G$, $H$, $\theta$  in Theorem 7.12 of [T1], 
 correspond to $G$, $L$, $\tau$; while
 $G'$, $H'$, $\theta'$ correspond
 to $H$, $C$ and $\tau^*$, respectively.
 Now,  
 Theorem 7.12 (1) and (2) predicts a bijection
 $$^{\prime}: \bigcup_{x \in H}\irr{N|\tau^x} \rightarrow
 \bigcup_{x \in H}\irr{N\cap H|(\tau^*)^x} \, ,$$
 which commutes with the action of $H$ (part  (7)
 of Theorem 7.12). By parts (4), (1)   and (2) of the same
 theorem, writing $R=L$ and $S=N$, we have that $\gamma \in \irr{N|\tau}$
 if and only if $\gamma' \in \irr{N\cap H|\tau^*}$.
We call $\pi(N, \tau)$
the restriction of the map $^\prime$ to $\irr{N|\tau}$.
Part (b) follows from Theorem 10.1 of [T3].
  \end{proof}

 \medskip
 The following is easy.
 \medskip

\begin{lem}\label{easyglau}
Suppose that $LQ \nor G$, where $L \nor G$,
 $(|L|,|Q|)=1$, and $Q$ is a $p$-group for some prime $p$.
 Suppose that $\tau \in {\rm Irr}_Q(L)$, and let $\tau^* \in \irr{C}$
 be the Glauberman correspondent, where $C=\cent LQ$.
 Suppose that  $Z \nor G$
 is contained in $C$. Let $\lambda \in \irr Z$ be $L$-invariant.
 Let $H=\norm G Q$.
 Suppose that $$\lambda^L=f(\tau^{h_1} + \ldots +\tau^{h_s}) \, ,$$
 for some $h_i \in H$, and some integer $f$.
Then 
$$\lambda^C=f^*((\tau^*)^{h_1} + \ldots +(\tau^*)^{h_s}) \, ,$$
for some integer $f^*$.
\end{lem}

 \medskip
 \begin{proof}
 We know by Theorem (13.29) of [Is] that
 if $\nu \in {\rm Irr}_Q(L)$, then  $\nu^*$ lies
 above $\lambda$ if and only if  $\nu$ lies above $\lambda$.
 Let $\rho \in \irr{C|\lambda}$. Then $\rho=\nu^*$
 for some $\nu \in \irr{L|\lambda}$. Thus
 $\nu=\tau^h$ for some $h \in H$, by hypothesis. Then 
 $$\rho=\nu^*=(\tau^h)^*=(\tau^*)^h \, ,$$
 because $H$ commutes with 
 Glauberman  correspondence.
 Since $\lambda$ is $C$-invariant, then
 we easily conclude the proof of the lemma.
 \end{proof}

 \medskip




\medskip

\section{Theorem A}

In this section, we prove Theorem A.

  \medskip
\begin{thm}\label{main}
Assume Hypothesis \ref{hyp}.
Then $I_N(\lambda)/Z$ is solvable.
\end{thm}

\medskip

\begin{proof}
We argue by  induction  on $|N:Z|$.
 Let $S/Z$ be the largest solvable normal
subgroup of $N/Z$. Let $T=I_G(\lambda)$
be the stabilizer of $\lambda$ in $G$.

\smallskip
\noindent
{\sl Step 1.}~~{\sl We may assume that
 $\lambda$ is $G$-invariant.}
 \smallskip
 
We claim that $(I_G(\lambda), I_N(\lambda), \lambda)$
satisfies Hypothesis 3.1. Indeed,
let $\tau_i \in \irr{I_N(\lambda) |\lambda}$ for $i=1,2$.
  
 By Lemma \ref{propert}(c) (with $K=Z$),
  there is $g \in I_G(\lambda) $ such that
    $(\tau_1)^{g}=\tau_2$.
 Hence, by working in $I_G(\lambda)$, we see that
 it is no loss to
  assume that $\lambda$ is invariant in $G$.
 Hence, we wish to prove that $N/Z$ is solvable, that is, that $S=N$.
 
 \smallskip
\noindent
{\sl Step 2.}~~{\sl If $Z\le K<N$, with $K \nor G$, then $K/Z$
is solvable. Also $N/S$ is isomorphic
to a direct product of a non-abelian simple group $X$.}
 \smallskip
 
  By Lemma \ref{propert} (a) and induction, we have that
if $Z\le K<N$, with $K \nor G$, then $K/Z$
is solvable.   Then $N/S$
is a chief factor of $G/Z$,  and it is isomorphic
to a direct product of a non-abelian simple group $X$.

 \smallskip
\noindent
{\sl Step 3.}~~{\sl We may assume that $Z$ is cyclic and that
  $\lambda$
 is faithful.}
 \smallskip

  This follows by using character triple isomorphisms.  
 \smallskip

\noindent
{\sl Step 4.}~~{\sl If $Z <K \sbs N$ is a normal
 subgroup of $G$, and $\tau \in \irr{K|\lambda}$, then $I_N(\tau)/K$ 
is solvable. Also $S>Z$.}
 \smallskip

The first part is a  direct consequence of Lemma \ref{propert}(c) and induction.
 If $S=Z$, then by Step 2, we have that $N/Z$ is a minimal normal non-abelian
 subgroup of $G/Z$. 
 Then $N/Z$ is a direct product of non-abelian
 simple groups isomorphic to $X$, and $Z=\zent N$.
 Also, $N'Z=N$.
 By Theorem \ref{HI},
 there is a prime $p$ dividing $|X|$ such that
 $p$ does not divide $|M(X)|$. By Corollary 7.2 of [HI],
 we have that $p$ does not divide $|N'\cap Z|$.
 Since $p$ divides $|N'|$, 
 this contradicts   Theorem \ref{div} with $U=N'$.

 \smallskip
 \noindent
{\sl Step 5.}~~{\sl  $F={\bf F}(N)=S$.}
 \smallskip
 
 Otherwise, let  $R/F$ be a solvable chief factor
 of $G$ inside $N$. Thus $R/F$ is a $q$-group
 for some prime $q$. 
Let $L$ be the Sylow $q$-complement
of $F$. Let $Z_{q'}=L \cap Z$.
Let  $Q$ be a Sylow $q$-subgroup of $R$,
so that $R=LQ$. Let $Z_q=Q\cap Z$,
so that $Z=Z_{q'} \times Z_q$.
We have that $G=LH$,
where $H=\norm G Q$, by the Frattini argument.
Let $C=\cent LQ$.

Write $\lambda=\lambda_{q'} \times \lambda_{q}$,
where $\lambda_{q'}=\lambda_{Z_{q'}}$, and
$\lambda_q=\lambda_{Z_q}$.
By coprime action and counting, we see that 
$Q$ fixes some
$\tau_{q'} \in \irr{L|\lambda_{q'}}$. 
Let $\tau=\tau_{q'} \times \lambda_{q} \in \irr{LZ}$.
By hypothesis and Lemma \ref{propert}(a), we can write
$$\lambda^{LZ}=f( \tau^{h_1} + \ldots + \tau^{h_s}) \, ,$$
where $h_i \in H$, and $\lambda^{h_i}=\lambda$, because
$\lambda$ is $G$-invariant. 
Hence
$$\lambda_{q'}^L=f( \tau_{q'}^{h_1} + \ldots + \tau_{q'}^{h_s})\, .$$
By  Lemma \ref{easyglau},
we have that 
$$\lambda_{q'}^C=f^*( (\tau_{q'}^*)^{h_1} + \ldots + (\tau_{q'}^*)^{h_s})\, .$$
 
By   Theorem \ref{deepglau},
we know that there is a bijection
$$\pi(N, \tau_{q'}): \irr{N|\tau_{q'}) \rightarrow \irr{\norm NQ|\tau_{q'}^*}}$$
that commutes with $H $-action.

We claim that $(\norm GQ, \norm NQ, \lambda)$
satisfies Hypothesis \ref{hyp}. If this is the case,
since $\norm NQ < N$, we will have that 
$|\norm NQ : Z|<|N:Z|$, and by induction,
we will conclude that $\norm NQ/Z$ is solvable.
This implies that $N/Z$ is solvable,
and the proof of the theorem would be complete.  Suppose now that  
 $\psi_i \in \irr{\norm NQ |\lambda}$ for $i=1,2$.
 We are going to show that   there exists $x \in H$ 
such that $\psi_1^x=\psi_2$. 
Since $\psi_i$ lies over $\lambda_{q'}$,
then we have that $\psi_1$ lies over
some $(\tau_{q'}^*)^{h_j}$, and $\psi_2$ lies over
some $(\tau_{q'}^*)^{h_k}$ for some $h_j, h_k \in H$. 
Conjugating by $h_j^{-1}$ and by $h_k^{-1}$, we may assume that
$\psi_1$ and $\psi_2$ lie over $\tau_{q'}^*$.

Now, we know that there exists $\mu_i \in \irr{N|\tau_{q'}}$
such that $\pi(N, \tau_{q'})(\mu_i)=\psi_i$. 
In fact, since $\psi_i$ lies over $\lambda_q$, 
we have that $\mu_i \in \irr{N|\lambda_q}$ by Theorem \ref{deepglau}(b) (with the role
of $\lambda$ in that theorem being played now here by $\lambda_q$),
and therefore $\mu_i \in \irr{N|\tau} \sbs \irr{N|\lambda}$.
By hypothesis, there is $h \in H $ such that
$\mu_1^h=\mu_2$.
Now, $\tau_{q'}^h$ and $\tau_{q'}$ are below $\mu_2$,
so there is $h_1 \in N\cap H$ such that
$\tau_{q'}^{hh_1}=\tau_{q'}$. Replacing $h$
by $hh_1$, we may assume that $(\tau_{q'})^h=\tau_{q'}$.  
Now 
$$\psi_1^h=\pi(N, \tau_{q'})(\mu_1)^h=
\pi(N, \tau_{q'}^h)(\mu_1^h)=\pi(N, \tau_{q'})(\mu_2)=
\psi_2 \,,$$
as desired.
By induction, $N\cap H$ is solvable, so $N$ is solvable.
This proves Step 5.

\medskip

 \smallskip
 \noindent
{\sl Step 6.}~~{\sl  If $p$ divides $|F:Z|$, then $N$
has a solvable subgroup of $p$-power index. Therefore, so do
  the simple groups factors in the direct product of $N/S$.}
 \smallskip

Suppose that $Q/Z$ is a non-trivial normal $p$-subgroup
of $G/Z$, where $Q$ is contained in $N$. 
Then the irreducible constituents of $\lambda^Q$ all have the same degree
by Lemma \ref{propert}(a), for instance.
So we can write
$$\lambda^Q=f(\tau_1 +  \cdots + \tau_k) \, , $$
where $\tau_i \in \irr{Q |\lambda}$ are all the different
constituents. Write $\tau=\tau_1$. Notice that $f=\tau(1)$. 
Thus we deduce that $k$ is a power of $p$.
Now, since $G$ acts on $\Omega=\{\tau_1, \ldots, \tau_k\}$
transitively  by conjugation by Lemma 3.2(a),
we have that $|G:I_G(\tau)|=k$ is a power of $p$.  Hence, $|N:I_N(\tau)|$ is a power
of $p$. If $Q>Z$, then we know by induction
that $I_N(\tau)/Q$ is solvable. In this case, we deduce that
that $N$ has a solvable subgroup
with $p$-power index.  The same happens for factors of $N$.

 \smallskip
 \noindent
{\sl Step 7.}~~{\sl  Final contradiction.}
 \smallskip

We know by Step 2 that $N/S$ is isomorphic
to a direct product of  a non-abelian simple group $X$.
By Theorem (2.1), there exists a prime $q$ dividing $|X|$,
such that $q$ does not divide the order of the Schur multiplier of $X$,
and such that no solvable subgroup of $X$ has $q$-power index.
By Step 6, we have that $q$ does not divide $|F:Z|$.
Let $W$ be the normal $q$-complement of $F$. Hence $F=WZ$.
Also $F/W=\zent{N/W}$. By Corollary 7.2 of [HI], we have that
$q$ does not divide $|(N/W)' \cap F/W|$.
But $F/W$ is a $q$-group, so $(N/W)' \cap F/W=W/W$.
In particular, $N' \cap F \sbs W$.
Thus $q$ does not divide $|N' \cap F|$. Thus $q$ does not
divide $|N' \cap Z|$. Since $N/F$ is perfect, we have that $N' F=N$, so that $q$
divides $|N'|$.  
But this contradicts Theorem \ref{div} with $U=N'$. 
 \end{proof}

\medskip
Next is Theorem A of the introduction.
 \begin{cor}
 Suppose that $Z \nor G$, and let $\lambda \in \irr Z$.
  Assume that if $\chi, \psi \in \irr{G|\lambda}$,
  then there exists $a \in {\rm Aut}(G)$ stabilizing $Z$
  such that $\chi^a=\psi$. 
  If $T$ is the stabilizer of $\lambda$ in $G$, then
  $T/Z$ is solvable.
 \end{cor}
 
  \begin{proof}
 Let $A={\rm Aut}(G)_Z$ be the group of automorphisms
 of $G$ that stabilize $Z$.
 Let $\Gamma=GA$ be the semidirect product.
 We have that $Z \nor  \Gamma$.
 By hypothesis, $(\Gamma, G, \lambda)$ satisfies Hypothesis \ref{hyp}.
 By Theorem \ref{main}, we have that $T/Z$
 is solvable.
\end{proof} 
 \medskip
 
 \section{Theorem B}
 
 \noindent

 We begin by giving another proof of a result of U. Riese ({\bf [5]})
that we shall need later. \medskip

\begin{lem}
Let $H \sbs G$ and $\alpha \in \irr H$.
Suppose that $\alpha^G = \chi \in \irr G$ and that every irreducible
constituent of $\chi_H$ has degree equal to $\alpha(1)$. Then
$\chi$ vanishes on $G - H$.
\end{lem}
\medskip

\begin{proof}
By hypothesis, $\chi_H$ is the sum of
$\chi(1)/\alpha(1) = |G:H|$ irreducible characters, and thus
$[\chi_H,\chi_H] \ge |G:H|$. Then
$|H|[\chi_H,\chi_H] \ge |G|[\chi,\chi]$, and so $\chi$ vanishes on
$G - H$, as claimed. 
\end{proof}
\medskip

Next is the proof of Riese's theorem (by M. Isaacs).
\medskip

\begin{thm} \label{rie}
Let $A \sbs G$, where $A$ is abelian, and
assume that $\lambda^G$ is irreducible, where
$\lambda \in \irr A$. Then $A \nor\nor G$.
\end{thm}
\medskip

\begin{proof}
We prove the theorem by induction on $|G|$.
Write $\chi = \lambda^G \in \irr G$, and let $V$ be   the subgroup of $G$ generated by the
elements $g\in G$ with $\chi(g)\ne 0$. 
Since $A$ is abelian, each irreducible
constituent of $\chi_A$ has degree $1 = \lambda(1)$, and thus by
Lemma 6.1, we have $V \sbs A$. Also, writing $Z = \zent G$, we have
$Z \sbs V$.

 If $A \sbs H < G$, then since
$\lambda^H$ is irreducible, the inductive hypothesis yields
$A \nor\nor H$. Assuming that $A$ is not subnormal in $G$, then
Wielandt's zipper
lemma (Theorem 2.9 of \cite{Is4}) guarantees that there is a unique maximal subgroup $M$ of
$G$ with $A \sbs M$. Also, if the normal closure $A^G < G$ then $A \nor\nor A^G \nor G$,
and we are done. We can thus supppose that $A^G = G$, and so
$A^g \not\sbs M$ for some element $g \in G$. By the uniqueness of
$M$, therefore, we have $<A, A^g> = G$. But $V \nor G$ and
$V \sbs A$, and thus $V \sbs A \cap A^g \sbs Z$, and we have
$V = Z = A \cap A^g$. Thus $\chi$ vanishes off $Z$, and so $\chi$ is
fully ramified with respect to $Z$. In particular
$|G:A|^2 = \chi(1)^2 = |G:Z|$, and we have $|G:A| = |A:Z|$. Thus
$|G:A| = |A^g:A^g \cap A|$, and it follows that $AA^g = G$. This
implies that $A = A^g$, and thus $A = G$. This is a contradiction
since $A$ was assumed to be not subnormal.
\end{proof}
\medskip

\begin{cor}
Let $\theta \in \irr N$, where
$N \nor G$ and $\theta$ is $G$-invariant. Let $N \sbs A \sbs G$,
where $A/N$ is abelian, and suppose that $\theta$ has an
extension $\phi \in \irr A$ such that $\phi^G$ is irreducible. Then
$A \nor\nor G$.
\end{cor}
\medskip

\begin{proof}
By character triple isomorphisms (see Chapter 11 of [Is]), we can assume
that $\theta$ is linear and faithful. Then $\phi$ is linear and
$A' \sbs N \cap \ker\phi = \ker\theta = 1$. Then $A$ is abelian, and
since $\phi^G$ is irreducible, Theorem \ref{rie} yields the result.
\end{proof}
\medskip

Now, we prove an extension of Theorem B.  We should mention that
 the Classification of Finite Simple Groups
is implicitely used in the proof. Specifically,  in the Howlett-Isaacs
([HI])
theorem on central type groups.

\begin{thm}
Let $N \nor G$. Suppose that
$\theta \in \irr N$ is $G$-invariant and that $o(\theta)\theta(1)$ is
a $\pi$-number. Assume that $G/N$ is $\pi$-separable and that
$\oh\pi{G/N} = 1$.  Then all members of $\irr{G|\theta}$ have equal
degrees if and only if $G/N$ is an abelian $\pi'$-group.
\end{thm}

\medskip

\begin{proof}
If $G/N$
is an abelian $\pi'$-group, then $\theta$ extends
to $G$ by Corollary 8.16 of [Is], and
we are done by Gallagher's  Corollary 6.17
of [Is]. To
prove the converse, we argue by induction 
 on $|G/N|$ and assume that 
$|G/N| > 1$. We argue first that the common degree $d$ of the
characters in $\irr{G|\theta}$ is a $\pi$-number. To see this, let
$q \in \pi'$ and let $Q/N \in \syl q{G/N}$. Then $\theta$ extends to
$Q$, and the induction to $G$ of such an extension has degree
$\theta(1)|G:Q|$, which is a $q'$-number. Since this degree is a
multiple of $d$, it follows that $d$ is a $q'$-number, and since
$q \in \pi'$ was arbitrary, we see that $d$ is a $\pi$-number.

Let $U/N = \oh{\pi'}{G/N}$ and note that $U > N$. All degrees of
characters in $\irr{U|\theta}$ divide $d$, and so are
$\pi$-numbers. But since $U/N$ is a $\pi'$-number, it follows that
all degrees of characters in $\irr{U|\theta}$ equal $\theta(1)$, and
so all of these characters extend $\theta$. It follows that $U/N$ is
abelian by Gallagher Corollary 6.17
of [Is]. If $U = G$, we are done, and so we suppose that $U < G$
and we let $V/U = \oh\pi{G/U}$. Note that $V > U$. By
Corollary 8.16 of [Is], there exists a unique extension $\hat\theta \in \irr U$
of $\theta$ with determinantal $\pi$-order. By uniqueness,
$\hat\theta $ is $G$-invariant. Now, let $\phi \in  \irr{V|\hat\theta }$.
Since $V/U$ is a $\pi$-group, $\phi_U$
is a multiple of $\hat\theta $  and $o(\hat\theta )$ is a $\pi$-number,
we easily have that $o(\phi)$ is a $\pi$-number. 
Write $T = G_\phi$ for the stabilizer of $\phi$
in $G$. Then all
members of $\irr{T|\phi}$ induce irreducibly to $G$, yielding
characters of degree $d$, and thus these characters all have
degree $d/|G:T|$. We claim that $T$ satisfies the hypotheses of
the theorem with respect to the character $\phi$ and the normal
subgroup $V$. To see this, we need to check that
$\oh\pi{T/V}$ is trivial.

Let $W/V = \oh{\pi'}{G/V}$. We argue that $W$ stabilizes
$\phi$. This is because the $G/V$-orbit of $\phi$ has size
dividing $d$, and so is a $\pi$-number, and $W/V$ is a normal
$\pi'$-subgroup of $G/V$. Thus $W \sbs T$ and $\oh\pi{T/V}$
centralizes the normal $\pi'$-subgroup $W/V = \oh{\pi'}{G/V}$.
But $\oh\pi{G/V}$ is trivial, and Lemma~1.2.3 applies to show
that $\oh\pi{T/V} = 1$, as wanted.

By the inductive hypothesis, we conclude that $T/V$ is a
$\pi'$-group. Also, by the Clifford correspondence, $|G:T|$ divides
$d$, which we know is a $\pi$-number. Thus $T/V$ is a full Hall
$\pi'$-subgroup of $G/V$. Also, $\phi$ extends to $T$, and so
$\phi(1) = d/|G:T| = d/|G/V|_{\pi}$ is constant for
$\phi \in \irr{V|\theta}$. It follows that the hypotheses are
satisfied in the group $V$ with respect to $\theta$. If $V < G$,
the inductive hypothesis yields that $V/N$ is a $\pi'$-group, and
  this is a contradiction.

It follows that   $V = G$ and $G/U$ is a
$\pi$-group. Also, $G/U$ acts faithfully on $U/N$ because
$\oh\pi{G/N}$ is trivial. Now  let $\lambda \in \irr{U/N}$, so that $\lambda$ is linear. Let
$S = G_\lambda$, and note that $\lambda$ extends to $S$ since
$S/U$ is a $\pi$-group. Write $a = |G:S|$.

Note that $S$ is the stabilizer of $\lambda\hat\theta$ in $G$, and
thus all characters in $\irr{S|\lambda\hat\theta}$ have degree
$d/a$. If $r$ is the number of such characters, this yields
$r(d/a)^2 = |S:U|\theta(1)^2$. Also, since $\lambda$ extends to
$S$, by Theorem 
6.16 of [Is]
there is a degree-preserving bijection between
$\irr{S|\lambda\hat\theta}$ and $\irr{S|\hat\theta}$, and hence the
latter set contains exactly $r$ characters, and each has degree
$d/a$. Each of these must therefore induce irreducibly to $G$, and
it follows that each member of $\irr{G|\hat\theta}$ is induced from
a member of $\irr{S|\hat\theta}$.

 Note that the number of
different members of $\irr{S|\hat\theta}$ that can have the same
induction to $G$ is at most $|G:S| = a$.

Now let $t = |\irr{G|\hat\theta}|$ so that $td^2 = |G:U|\theta(1)$.
If we divide this equation by our previous one, we get
$ta^2/r = |G:S| = a$, and so $t = r/a$. It follows that each of the
$t$ members of $\irr{G|\hat\theta}$ is induced from exactly $a$
characters in $\irr{S|\hat\theta}$. In other words, if
$\chi\in\irr{G|\hat\theta}$, then $\chi_S$ has exactly $a$ distinct
irreducible constituents, each with degree $d/a$, and so by
Lemma 6.1, it follows that $\chi$ vanishes on $G-S$. In other words,
the only elements of $G$ on which $\chi$ can have a nonzero value
lie in the stabilizer of $\lambda$ for every linear character
$\lambda$ of $U/N$. But $G/U$ acts faithfully on this set of linear
chararacters, and thus $\chi$ vanishes on $G-U$. In other words,
$\hat\theta$ is fully ramified in $G$. It follows that
$d = \theta(1)|G:U|^{1/2}$. 

Also, $a\theta(1)$ divides $d$, and so $a$ must divide
$|G:U|^{1/2}$. Write $s = |S:U|$, so that $as = |G:U|$. Then
$a^2$ divides $as$, and thus $a$ divides $s$. In particular, we have
$a \le s$, so $|G:S| \le |S:U|$.
Thus 
$$|G:U|=|G:S||S:U| \le |S:U|^2 \, .$$

Now, by the Howlett-Isaacs theorem we have that $G/U$ is solvable. This
group acts faithfully on the group of linear characters of $U/N$,
and so by the main result in [D], there exist character stabilizers $T$ and $R$ such
that $T \cap R = U$. By the result of the previous paragraph, each
of $T/U$ and $R/U$ has order at least $|G:U|^{1/2}$.
Now
$$|G:U|=|G:T||T:U|\ge |R:U||T:U|\ge |G:U| \, .$$ Then
$TR = G$, and that each of $|T:U|$ and $|R:U|$ has order
$|G:U|^{1/2}$. Therefore all characters in $\irr{T|\hat\theta}$ are
extensions of $\hat\theta$ and induce irreducibly to $G$. In
particular, $T/U$ is abelian, and similarly $R/U$ is abelian.

By Corollary 6.3, it follows that $R \nor\nor G$, and since $R/U$ is
abelian, $R/U \sbs {\bf F}(G/U)$. Similarly, $T/U \sbs {\bf F}(G/U)$ and
thus $G/U$ is nilpotent. But then, since $G/U$ acts faithfully on
the group of linear characters of $U/N$, it follows that if $G/U$ is
nontrivial, then some character $\lambda \in \irr{U/N}$ has a
stabilizer $S$ in $G$  such that $$|S:U| < |G:U|^{1/2}$$ by Theorem B of [Is3]. But then
$|G:U|=|G:S||S:U|\le |S:U|^2<|G:U|$.
This contradiction completes the proof.
\end{proof}

\medskip

\noindent
{\bf Acknowledgment.}~~We thank the referee for a very careful reading of this manuscript.

\end{document}